\documentclass[12pt, a4paper]{amsart}

\usepackage{amsmath}
\usepackage{amssymb}
\usepackage{amsthm}
\usepackage{geometry}
\usepackage{hyperref}
\usepackage{enumitem} % For custom list labels (e.g., CS1, WP1)
\newcommand{\pos}{\mathbb{R}_+} % {{|R}_+}
\newcommand{\ca}{\mathcal{A}}
\newcommand{\cx}{\mathcal{X}}
\newcommand{\ct}{\mathcal{T}}

\newcommand{\cu}{\mathcal{U}}
\newcommand{\cp}{\mathcal{P}}

\newcommand{\cd}{\mathcal{D}}
\newcommand{\ch}{\mathcal{H}}
\newcommand{\crl}{\mathcal{R}_\lambda}
\newcommand{\ttt}{(\mathcal{T}(t))_{t\geq 0}}
\newcommand{\doX}{\partial\Omega} % \do
\newcommand{\dx}{\partial X}
\newcommand{\limn}{\lim\limits_{n\rightarrow\infty}}
\newcommand{\myVert}{\Vert} % Original \Arrowvert, using standard LaTeX norm symbol
\theoremstyle{definition}
\newtheorem{definition}{Definition}[section]
\newtheorem{example}[definition]{Example}
\newtheorem{remark}[definition]{Remark}
\newtheorem{program}[definition]{Program}
\newtheorem{assumption}[definition]{Assumption}

\theoremstyle{plain}
\newtheorem{proposition}[definition]{Proposition}
\newtheorem{lemma}[definition]{Lemma}
\newtheorem{theorem}[definition]{Theorem}
\newtheorem{corollary}[definition]{Corollary}

\numberwithin{equation}{section}
\begin{document}

\title{SEMIGROUPS FOR INITIAL--BOUNDARY VALUE PROBLEMS}

 \author{Marjeta Kramar}
\address{Marjeta Kramar\\
University of Ljubljana\\
Faculty of Civil and Geodetic Engineering\\
Department for Mathematics and Physics\\
Jamova 2\\
1000 Ljubljana, Slovenia}
\email{mkramar@fgg.uni-lj.si}

 \author{Delio Mugnolo}
\address{Delio Mugnolo\\
Arbeitsbereich Funktionalanalysis\\
Mathematisches Institut\\
Universität Tübingen\\
Auf Der Morgenstelle 10\\
D- 72076 Tübingen, Germany}
\email{demu@fa.uni-tuebingen. de}

 \author{Rainer Nagel}
\address{Rainer Nagel\\
Arbeitsbereich Funktionalanalysis\\
Mathematisches Institut\\
Universität Tübingen\\
Auf Der Morgenstelle 10\\
D- 72076 Tübingen, Germany}
\email{rana@fa.uni-tuebingen.de}

% Footnote for MSC and acknowledgments



\thanks{We thank K.-J. Engel for many helpful comments. \\
  The second author is supported by the Istituto Nazionale di Alta Matematica ``Francesco Severi''.\\[3pt]
  This article  was originally published in: M.\ Iannelli, G.\ Lumer (eds), Evolution Equations: Applications to Physics, Industry, Life Sciences and Economics. Progress in Nonlinear Differential Equations and Their Applications, vol 55. Birkhäuser, Basel. 2003, 275--292.
  }

%\begin{abstract}
%We motivate and explain the use of product spaces in order to treat abstract and concrete initial boundary value problems within a semigroup setting. We prove well-posedness under quite general conditions and treat boundary feedbacks using the theory of one-sided coupled operator matrices due to K.-J. Engel.
%\end{abstract}
\maketitle

% --- SECTION 1 ---
\section{SEMIGROUPS EVERYWHERE}
\label{sec:1}

In the foreword of what later became the classic in semigroup theory, E. Hille [Hi48] wrote in 1948:

\textit{``The analytical theory of semi-groups is a recent addition to the ever-growing list of mathematical disciplines. It was my good fortune to take an early interest in this discipline and to see it reach maturity. It has been a pleasant association: I hail a semi-group when I see one and}
\begin{center}
\textbf{I seem to see them everywhere}\textit{!}''
\end{center}

With this quite general (and provocative) statement he probably expressed his conviction that behind every parabolic or hyperbolic (linear and autonomous) partial differential equation there is a \textit{semigroup} yielding its solutions. In the meantime this has been confirmed in many ways and is presented in excellent form, e.g. in \cite{[Go85]} or \cite{[Pa83]}. We briefly recall one of the standard examples.

\begin{example}\label{exa:1.1}
Let $\Omega$ be a bounded open domain in $\mathbb{R}^n$ with smooth boundary $\partial\Omega$. The heat equation with Dirichlet boundary conditions is
\begin{equation} \tag{HE}\label{eq:HE}
\left\{
\begin{split}
 {\partial\over\partial t}u(t,x)&=\sum_{i=1}^n{\partial^2
\over\partial x_i^2}u(t,x),\qquad &&t\geq0,\thinspace
x=(x_1,\ldots,x_n)\in\Omega,\\
 u(t,z)&=0, &&t\geq0,\thinspace z\in\doX,\\
 u(0,x)&=f(x), &&x\in\Omega.
\end{split}
\right.
\end{equation}

If we now choose an appropriate function space, e.g., $X:=L^2(\Omega)$
and define an operator
$$Au:=\Delta u:=\sum_{i=1}^n{\partial^2\over\partial x_i^2}u\quad\hbox{with domain}\quad D(A):=H^2(\Omega)\cap H^1_0(\Omega),$$
the system \eqref{eq:HE} can then be re-written as an abstract Cauchy problem
\begin{equation} \tag{ACP}\label{eq:ACP}
\left\{
\begin{split}
 \dot{u}(t)&=Au(t), \qquad &&t\geq0,\\
 u(0)&=f, &&x\in\Omega.
\end{split}
\right.
\end{equation}

It is well-known (see, e.g., \cite[Thm.~7.2.7]{[Pa83]}) that $(A,D(A))$ 
generates a strongly continuous (and analytic) semigroup
$(T(t))_{{t\geq 0}}$
%$(T(t))_{t\geq 0}$ 
on $X$ such that
$$u(t,x):=T(t)f(x),\qquad t\geq0,\thinspace x\in\Omega,$$
are the (mild) solutions of \eqref{eq:HE}.
\end{example}
Once a partial differential equation is solved in this way by a
semigroup, there exist by now powerful tools to describe the
qualitative behavior of the solutions. We only mention the spectral
theory for semigroups (see \cite[Chap.~IV]{[EN00]}) yielding various
Liapunov-type stability theorems or the recent results on maximal
regularity of the inhomogeneous version of \eqref{eq:ACP} (see, e.g.,
\cite{[DHP01]}, \cite{[We01]}, or \cite{[We01b]}).

On the other hand, even when one admits the conceptual lucidity and
the generality of the semigroup approach, scepticism may remain, and
Hille continues his above quote:

\textit{``Friends have observed, however, that there are mathematical objects which are not semi-groups.''}

Indeed, already for our heat equation \eqref{eq:HE} semigroups seem to be inappropriate as soon as we add an inhomogeneous term on the boundary.

\begin{example}\label{exa:1.2.}
For the heat equation on $\Omega$ as in Example~1.1, 
we consider inhomogeneous boundary values given by
$u(t,z)=\phi(t,z)$ for $t\geq 0$, $z\in\doX$. As we will see
below, it is more convenient to write this in the form
\begin{equation}\tag{iHE} \label{eq:iHE}
\left\lbrace\begin{aligned}
 {\partial\over\partial t}{u}(t,x)&=\sum_{i=1}^n{\partial^2
\over\partial x_i^2}u(t,x),
\qquad t\geq0,\thinspace x=(x_1,\ldots,x_n)\in\Omega,\\
 {\partial\over\partial t}u(t,z)&=\psi(t,z),
\mkern 98mu t\geq0,\thinspace z\in\partial\Omega,\\
 u(0,x)&=f(x),\mkern111mu x\in\Omega,\\
u(0,z)&=g(z),\mkern113mu z\in\doX,
\end{aligned}\right.
\end{equation}
where $\psi(t,z):={\partial\over\partial t}\phi(t,z)$.\qed
\end{example}

Such a situation typically occurs in all kinds of boundary control
problems (see, e.g., \cite{[LT00]}). The obstacle to use semigroups consists
in the fact that the inhomogeneous term
$$t\mapsto\psi(t)$$
does not map into the state space chosen in Example~1.1. However, as for many other types of equations (e.g., delay equations, integro-differential equations) initially not having the form \eqref{eq:ACP}, it is possible to extend the state space and then write \eqref{eq:iHE} as an inhomogeneous abstract Cauchy problem:
take $X:=L^2(\Omega)$, $\dx:=L^2(\doX)$, and $A:=\Delta$ the
Laplacian with appropriate domain (see \S~3 below for details).
On the product space
$$\cx:=X\times\dx,$$
we define a new operator
\begin{equation}\label{eq:1.1}
\ca:=\begin{pmatrix}A&0\\ 0&0\end{pmatrix},\qquad
D(\ca):=\left\{{u\choose v}\in D(A)\times\dx\:u\big|_{\doX}=
v\right\}.
\end{equation}

Then \eqref{eq:iHE} becomes
\begin{equation}\tag{iACP} \label{eq:iACP}
\left\lbrace\begin{aligned}
 &\dot{\cu}(t)=\ca\;\cu(t)+{0\choose{\psi}(t)},\qquad
t\geq0,\\
 &\cu(0)=\cu_0,
\end{aligned}\right.
\end{equation}
for a function $\pos\ni t\mapsto\cu(t)\in\cx$ and with the initial
data $\cu_0:={f\choose g}$. Hence, it becomes an inhomogeneous abstract Cauchy problem and therefore semigroups should be applicable in this situation.

\begin{remark}\label{rem:1.3.}
The idea to use a product space in order to convert inhomogeneous boundary conditions into the form \eqref{eq:iACP} appears and is used systematically by Arendt in \cite{[Ar00]}, see also \cite[Chap.~6]{[ABHN01]}.
\end{remark}

This important example, and the quotation of Hille's cited at the
beginning, suggest the following.

\begin{program}\label{prog:1.4.}$ $\\[-10pt]
\begin{enumerate}[label=(\roman*)]
\item Define an abstract setting for the situation described in Example~\ref{exa:1.2.}.
\item Show that well-posedness becomes equivalent to the well-posedness of the corresponding \eqref{eq:ACP}.
\item Apply semigroup theory to obtain existence of the solutions.
\end{enumerate}
\end{program}

In the following, we propose an appropriate abstract setting, develop
this program, and discuss one typical application.

% --- SECTION 2 ---
\section{ABSTRACT INITIAL--BOUNDARY VALUE PROBLEMS}
\label{sec:2}

Let $X$ and $\dx$ be Banach spaces called \textit{state space} and \textit{boundary space}, respectively. We denote by ${\cx}:=X\times\dx$ their product space, and by $\pi_1$ and $\pi_2$ the projections onto $X$ and $\dx$, respectively.

\begin{assumption}\label{ass:2.1.}
We consider the following (linear) operators.
\begin{itemize}
\item $A:D(A)\subseteq X\rightarrow X$, called \textit{maximal operator}.
\item $L:D(A)\rightarrow\partial X$, called \textit{boundary operator}.
\item $B:D(B)\subseteq X\rightarrow\dx$, called \textit{feedback operator}.
\end{itemize}
\end{assumption}

For these operators, we assume that $D(A)\subseteq D(B)$ and
consider what we call an \textit{abstract initial--boundary value
problem on the state space $X$ and the boundary space $\dx$}:
\begin{equation}\tag{AIBVP} \label{eq:AIBVP}
\left\lbrace\begin{aligned}
 {d\over dt}{u}(t)&=Au(t),\qquad t\geq 0,\\
 {d\over dt}Lu(t)&=Bu(t),\qquad t\geq 0,\\
  u(0)&=f,\\
 Lu(0)&=g,
\end{aligned}\right.
\end{equation}
where $f\in X$, $g\in\dx$. (If we want to emphasize the
dependence on the initial data $f$, $g$, we write $\mathrm{(AIBVP)}_{f,g}$).

We now make precise what we understand by a solution of \eqref{eq:AIBVP}.

\begin{definition}\label{def:2.2}
Let $(f,g)\in D(A)\times \partial X$ be given initial data. A \textbf{classical solution} to $\mathrm{(AIBVP)}_{f,g}$ is a function $u$ such that
\begin{enumerate}[label=\rm{(CS$_{\arabic*}$)}]
\item $u\in C^1(\pos,X)$ and $u(t)\in D(A)$ for all $t\geq 0$,
\item $Lu\in C^1(\pos,\dx)$,
\item $u$ satisfies $\mathrm{(AIBVP)}_{f,g}$.
\end{enumerate}
Moreover, $\mathrm{(AIBVP)}$ is called \textbf{ well-posed} if
\begin{enumerate}[label=\rm{(WP$_{\arabic*}$)}, itemsep=2pt, parsep=0pt]
\item $\overline{D(A)}=X$,
\item $\mathrm{(AIBVP)}_{f,g}$ admits a unique classical solution $u(\cdot,f,g)$ for all $(f,g)\in D(A)\times\dx$ satisfying $Lf=g$,
\item for all sequences of initial data $(f_n,g_n)_{n\in\mathbb{N}}\subseteq D(A)\times \dx$ tending to $0$ and such that $Lf_n=g_n$, one has $\limn u(t,f_n,g_n)=0$ and $\limn Lu(t,f_n,g_n)=0$ uniformly for $t$ on compact intervals.
\end{enumerate}
A \textbf{mild solution} to $\mathrm{(AIBVP)}_{f,g}$ for arbitrary $(f,g)\in\mathcal X$ is a function $u$ such that
\begin{enumerate}[label=\rm{(MS$_{\arabic*}$)}, itemsep=2pt, parsep=0pt]
\item $u\in C(\pos,X)$ and $\int_0^tu(s)ds\in D(A)$ for all $t\geq0$,
\item $B\left(\int_0^\cdot u(s)ds\right)\in C(\pos,\dx)$,
\item $u$ satisfies the integrated problem
$$\left\lbrace\begin{aligned}u(t)&=f+A\int_0^t\!\!\!u(s)ds,\mkern114mu
t\geq0,\\
L\int_0^t\!\!\!u(s)ds
&=tg+\int_0^t\!\!\!B\left(\int_0^s\!\!\!u(r)\;dr\right)ds
,\qquad t\geq0.
\end{aligned}\right.
$$
\end{enumerate}
\end{definition}

In order to treat \eqref{eq:AIBVP} using semigroups we consider the
operator matrices on $\cx$ given by
\begin{equation}\label{eq:2.1}
\ca:=\begin{pmatrix}A&0\\ 0&0\end{pmatrix}\qquad{\rm with\enspace
domain}\qquad D({\ca}):=\left\{{u\choose v}\in D(A)\times\dx\:
Lu=v\right\},
\end{equation}
and
\begin{equation}\label{eq:2.2}
\tilde\ca:=\begin{pmatrix}A&0\\ B&0\end{pmatrix}\qquad{\rm with\enspace
domain}\qquad D(\tilde\ca):=D(\ca).
\end{equation}
We briefly state some of their properties.

\begin{lemma}\label{lem:2.3.}
Let $B$ be relatively $A$-bounded. Then the operator $\tilde\ca$ is closed if and only if the operator $\binom{A}{L}:D(A)\subseteq X\rightarrow X\times\dx$ is closed.
\end{lemma}

\begin{proof}
Assume first that $\binom{A}{L}$ is a closed operator. Let
$$\left\{{u_n\choose v_n}\right\}_{n\in\mathbb{N}}\subseteq D({\ca}),\quad
\lim_{n\rightarrow\infty}{u_n\choose v_n}={u\choose v}, \quad{\rm
and}\quad\lim_{n\rightarrow\infty}\tilde\ca{u_n\choose v_n}=\limn
{Au_n\choose Bu_n} \negthinspace=\negthinspace{w\choose z}$$ for
some $u,w\in X$ and $v,z\in\partial X$. Since
$$\limn \binom{A}{L}u_n\negthinspace=
\negthinspace\limn \binom{Au_n}{v_n}\negthinspace=
\negthinspace{w\choose v}$$
and $\binom{A}{L}$ is closed, we obtain
$u\in D(A)$ and $Lu=v$, thus showing
that ${u\choose v}\in D({\tilde\ca})$. Furthermore, $Au=w$. Finally,
the continuity of $B$ with respect to the graph norm of $A$ implies
that $z=\limn Bu_n=Bu$, and hence $\tilde\ca{u\choose v}={w\choose
z}$.

Assume now $\tilde\ca$ to be closed. Let
$$\{u_n\}_{n\in\mathbb{N}}\subseteq D(A),\quad
\lim_{n\rightarrow\infty}u_n=u,\quad{\rm
and}\quad\lim_{n\rightarrow\infty} \binom{A}{L}u_n=\limn
\binom{Au_n}{Lu_n} \negthinspace=\negthinspace{w\choose v}$$ for
some $u,w\in X$, $v\in\dx$. It follows that $\{Bu_n\}_{n\in\mathbb{N}}$
converges to some $z\in\dx$, and that
$${u_n\choose Lu_n}_{n\in\mathbb{N}}\mkern-30mu\subseteq D(\ca),
\quad\lim_{n\rightarrow\infty}{u_n \choose Lu_n}={u\choose v},
\quad\lim_{n\rightarrow\infty}\tilde\ca{u_n
\choose Lu_n}=\lim_{n\rightarrow\infty}{Au_n\choose Bu_n}={w\choose
z},$$
and hence, due to closedness of $\tilde\ca$,
$${u\choose v}\in D(\ca)\quad{\rm and}\quad\tilde\ca{u\choose
v}={Au\choose Bu}={w\choose z},$$
i.e., $u\in D(A)$, $Lu=v$, and $Au=w$.
\end{proof}

\begin{lemma}\label{lem:2.4.}
Assume that $L$ is surjective and $\ker(L)$ is dense in $X$. Then the operator $\tilde\ca$ is densely defined.
\end{lemma}

\begin{proof}
Let $u\in X$, $v\in\partial X$, $\varepsilon>0$. Surjectivity of $L$ ensures that there exists $w\in D(A)$ such that $Lw=v$. Take $\tilde u,\tilde w\in\ker(L)$ such that $\myVert u-\tilde u\myVert_{_X}<\varepsilon$ and $\myVert w-\tilde w\myVert_{_X}<\varepsilon$. Let $z:=\tilde u+w-\tilde w\in D(A)$. Then
$$\bigg\myVert{u\choose v}-{z\choose v}\bigg\myVert
\le\bigg\myVert{u-\tilde u\choose0}\bigg\myVert+
\bigg\myVert{w-\tilde
w\choose0}\bigg\myVert\le2\varepsilon.$$
Since $L(z)=L(w)=v$, we obtain ${z\choose v}\in D(\tilde\ca) $.
\end{proof}

For the operator $\tilde\ca$, we consider the abstract Cauchy problem
on the product space $\cx$
\begin{equation}\tag{ACP$_\mathbf{x}$} \label{eq:ACP-bx}
\left\lbrace\begin{aligned}
\dot{\cu}(t)&=\tilde\ca\;\cu (t),\qquad
t\geq0,\\
{\cu}(0)&=\mathbf{x},
\end{aligned}\right.
\end{equation}
with initial data $\mathbf{x}\in\cx$.

As usual (cf. \cite[Def.~II.6.1]{[EN00]}), for $\mathbf{x}\in D(\tilde\ca)$ a
function $\cu$ is called \textbf{classical solution} of the abstract
Cauchy problem \eqref{eq:ACP-bx} if 
\begin{itemize}
\item $\cu\in C^1(\pos,\cx)$,
\item $\cu(t)\in D(\tilde\ca)$ for all $t\geq0$,
\item $\cu$ satisfies $\mathrm{(ACP)}_\mathbf{x}$.
\end{itemize}
As in \cite[Def.~II.6.8]{[EN00]}, the abstract Cauchy problem
$\mathrm{(ACP)}$ associated to the closed operator $\tilde\ca$ is called \textbf{well-posed} if
\begin{itemize}
\item $D({\tilde\ca})$ is dense in $\cx$,
\item $\mathrm{(ACP)}_\mathbf{x}$ admits a unique classical solution $\cu (\cdot,\mathbf{x})$ for all $\mathbf{x}\in D({\tilde\ca})$,
\item for every sequence of initial data $(\mathbf{x}_n)_{n\in\mathbb{N}}\subseteq D(\tilde\ca)$ tending to $0$ there holds $\limn {\cu}(t,\mathbf{x}_n)=0$ uniformly in $t$ on compact intervals.
\end{itemize}
Further (cf. \cite[Def.~II.6.3]{[EN00]}), for
arbitrary $\mathbf{x}\in\cx$ a function $\cu$ is called \textbf{mild
solution} of $\mathrm{(ACP)}_\mathbf{x}$ if
\begin{itemize}
\item $\cu\in C(\pos,\cx)$,
\item $\int_0^t\cu(s)\;ds\in D(\tilde\ca)$ for all $t\geq0$,
\item $\cu$ satisfies the integrated problem
$$\cu(t)=\mathbf{x}+\tilde\ca\int_0^t\!\!
\cu(s)\;ds,\qquad t\geq0.\eqno{\rm
(IP)}$$
\end{itemize}

We now apply the semigroup characterization of well-posedness of
abstract Cauchy problems (see \cite[Prop.~II.6.2, Prop.~II.6.4, Cor.~II.6.9]{[EN00]}).

\begin{proposition}\label{prop:2.5.}
Let $(\tilde\ca,D(\tilde\ca))$ be a closed operator on a Banach space $\cx$. Then \eqref{eq:ACP} is well-posed if and only if $\tilde\ca$ generates a strongly continuous semigroup. Moreover, if $(\tilde\ca,D(\tilde\ca))$ generates a strongly continuous semigroup $(\tilde\ct(t))_{t\geq0}$, then the function
$$\pos\ni t \mapsto \cu(t) := \tilde\ct(t)\mathbf{x}\in\cx\qquad\qquad
\hbox{for}\quad\mathbf{x}\in D(\tilde\ca)\quad(\mathbf{x}\in\cx,\ 
\hbox{respectively})$$
is the unique classical (mild, respectively) solution to
\eqref{eq:ACP-bx}.
\end{proposition}

We now show that well-posedness of \eqref{eq:AIBVP} is equivalent to
well-posedness of the corresponding $\mathrm{(ACP)}$. To that purpose we relate
the (classical and mild) solutions of the two problems.

\begin{lemma}\label{lem:2.6.}
\begin{enumerate}[label=(\roman*), itemsep=2pt, parsep=0pt]
\item If $u$ is a classical solution to $\mathrm{(AIBVP)}_{f,g}$, then ${\cu}:={u\choose Lu}$ is a classical solution to $\mathrm{(ACP)}_{\binom{f}{g}}$.
\item Conversely, if\enspace ${\cu}$ is a classical solution to $\mathrm{(ACP)}_\mathbf{x}$, then $u:=\pi_1{\cu}$ is a classical solution to $\mathrm{(AIBVP)}_{\pi_1(\mathbf{x}),\pi_2(\mathbf{x})}$.
\end{enumerate}
\end{lemma}

\begin{proof}
$(i)$ Let $u$ be a classical solution to $\mathrm{(AIBVP)}_{f,g}$.
It follows that $u\in C^1(\pos,X)$ and $Lu\in C^1(\pos,\dx)$, and
therefore ${\cu}\in C^1(\pos,\cx)$. Moreover, $Lf=g$ holds, hence
$\binom{f}{g}\in D({\tilde\ca})$, and also ${\cu}(t)\in
D({\tilde\ca})$ for all $t\geq0$. Finally, one can see that
$\mathrm{(ACP)}_{\binom{f}{g}}$ is fulfilled.

$(ii)$ Assume now ${u\choose v}:=\cu$ to be a classical solution to
$\mathrm{(ACP)}_\mathbf{x}$. Then ${\cu}(t)\in D(\tilde\ca)$ for all $t\geq 0$
and hence $u(t)\in D(A)$ and $v(t)=Lu(t)$ for all $t\geq 0$. It also
follows from $\cu\in C^1(\pos,\cx)$ that $u\in C^1(\pos,{X})$ and
$Lu\in C^1(\pos,\dx)$. Finally, $\pi_2(\mathbf{x})=v(0)=Lu(0)=L(\pi_1(\mathbf{x}))$,
and $\mathrm{(AIBVP)}_{\pi_1(\mathbf{x}),\pi_2(\mathbf{x})}$ is fulfilled.
\end{proof}

\begin{lemma}\label{lem:2.7.}
\begin{enumerate}[label=(\roman*), itemsep=2pt, parsep=0pt]
\item If $u$ is a mild solution to $\mathrm{(AIBVP)}_{f,g}$, then ${\cu}:={u\choose v}$ is a mild solution to $\mathrm{(ACP)}_{\binom{f}{g}}$, where $v(t):=g+B\left(\int_0^t u(s)ds\right)$, $t\geq0$.
\item Conversely, assume that\enspace ${\cu}$ is a mild solution to $\mathrm{(ACP)}_\mathbf{x}$. Then $u:=\pi_1{\cu}$ is a mild solution to $\mathrm{(AIBVP)}_{\pi_1(\mathbf{x}),\pi_2(\mathbf{x})}$.
\end{enumerate}
\end{lemma}

\begin{proof}
$(i)$ Let $u$ be a mild solution to $\mathrm{(AIBVP)}_{f,g}$. It follows
that $u\in C(\pos, X)$ and $B\left(\int_0^\cdot u(s)\right)ds
\in C(\pos,\dx)$, and therefore $\cu\in C(\pos,\cx)$. Moreover, $u$
fulfills the condition (MS$_3$), and this implies that
$L\int_0^tu(s)ds=\int_0^tv(s)ds$, i.e., $\int_0^t\cu(s)\;ds
\in D(\tilde\ca)$ for all $t\geq0$. Finally, there holds
$$u(t)=f+A\int_0^t\!\!\! u(s)\;ds,\qquad{\rm and}\qquad
v(t)=g+B\int_0^t\!\!\!
u(s)\;ds$$
(respectively, by (MS$_3$) and by definition). These equalities
express that (IP) is satisfied.

$(ii)$ Assume now ${u\choose v}:=\cu$ to be a mild solution to
$\mathrm{(ACP)}_\mathbf{x}$. It follows that $\cu\in C(\pos,\cx)$, and hence $u\in
C(\pos,X)$ and $B\left(\int_0^\cdot u(s)ds\right)\in
C(\pos,\dx)$. Moreover, $\int_0^t \cu(s)ds\in D(\tilde\ca)$ for
all $t\geq0$, and this yields
$$\int_0^t\!\!\!
u(s)\;ds\in D(A)\quad{\rm and}\quad
L\int_0^t\!\!\! u(s)\;ds=\int_0^t\!\!\!
v(s)\;ds=tg+\int_0^t\!\!\!B\left(\int_0^s\!\!\! u(r)\;dr\right)ds,$$
for all $t\geq0$.

Finally, taking the first coordinate of (IP), it follows that
$u(t)=\pi_1(\mathbf{x})+A\int_0^t u(s)ds$.
\end{proof}

\begin{theorem}\label{thm:2.8.}
\begin{enumerate}[label=(\roman*), itemsep=2pt, parsep=0pt]
\item If $({\tilde\ca},D({\tilde\ca}))$ generates a strongly continuous semigroup $(\tilde\ct(t))_{t\geq0}$ on $\cx$, then $\mathrm{(AIBVP)}$ is well-posed. In this case, the unique classical solution to $\mathrm{(AIBVP)}_{f,g}$ is given by $\pi_1{\tilde\ct}(\cdot)\binom{f}{g}$ for all initial data $(f,g)\in D(\tilde\ca)$.
\item Conversely, assume $({\tilde\ca},D({\tilde\ca}))$ to be closed and densely defined. If $\mathrm{(AIBVP)}$ is well-posed, then $({\tilde\ca},D({\tilde\ca}))$ generates a strongly continuous semigroup $(\tilde\ct(t))_{t\geq0}$ on $\mathcal X$.
\end{enumerate}
\end{theorem}

\begin{proof}
$(i)$ Let $({\tilde\ca},D({\tilde\ca}))$ generate a strongly continuous semigroup $({\tilde\ct}(t))_{t\geq0}$ on $\cx$. Then ${\tilde\ca}$ is densely defined, which implies that $A$ is densely defined, too. Moreover, $\tilde\ca$ is closed and, by Proposition~2.5, the associated $\mathrm{(ACP)}$ is well-posed. By Lemma~\ref{lem:2.6.}.$(ii)$, $u(\cdot,f,g)=\pi_1{\cu}(\cdot,\binom{f}{g})= \pi_1{\tilde\ct}(\cdot)\binom{f}{g}$ yields a classical solution to $\mathrm{(AIBVP)}_{f,g}$ for all $(f,g)\in D(A)\times\dx$ such that $Lf=g$, i.e., for all $(f,g)\in D(\tilde\ca)$. This classical solution is unique by Lemma~\ref{lem:2.6.}.$(i)$.

Let now $t_0>0$ and $(f_n,g_n)_{n\in\mathbb{N}}$ be a sequence of initial data satisfying $Lf_n=g_n$ and tending to 0. Note that $\binom{f_n}{g_n}\in D(\tilde\ca)$ and $\cu (t,\binom{f_n}{g_n})\in D(\tilde\ca)$ for all $t\geq0$. Hence we have $\limn {\cu}(t,\binom{f_n}{g_n})=0$ uniformly in $[0,t_0]$ if and only if $\limn u(t,f_n,g_n)=0$ and $\limn Lu(t,f_n,g_n)=0$ (both uniformly in $[0,t_0]$). Since $\mathrm{(ACP)}$ is well-posed, the assertion follows.

$(ii)$ Since $\tilde\ca$ is closed, it suffices by Proposition~\ref{prop:2.5.} to show that the associated $\mathrm{(ACP)}$ is well-posed. Let $\mathbf{x}\in D(\tilde\ca)$. Well-posedness of $\mathrm{(AIBVP)}$ yields, by Lemma~\ref{lem:2.6.}, existence and uniqueness of a classical solution to $(\mathrm{ACP})_\mathbf{x}$. By assumption, $\overline{D(\ca)}=\cx$, therefore it only remains to show the continuous dependence on initial data. Let $t_0>0$ and $(\mathbf{x}_n)_{n\in\mathbb{N}}\subseteq D({\ca})$ be a sequence of initial data tending to 0. Then,
$(\pi_1({\mathbf{x}}_n),\pi_2(\mathbf{x}_n))_{n\in\mathbb{N}}$ is a sequence of initial data for $\mathrm{(AIBVP)}$ tending to 0 and such that $L\pi_1({\mathbf{x}}_n)=\pi_2(\mathbf{x}_n)$, and, by assumptions, there holds $\limn u(t,\pi_1\mathbf{x}_n,\pi_2\mathbf{x}_n)=0$ and $\limn Lu(t,\pi_1{\mathbf{x}}_n,\pi_2\mathbf{x}_n)=0$ (both uniformly in $[0,t_0]$). Also, $\cu(\cdot,\mathbf{x}_n)=\binom{u}{Lu}(\cdot,{\pi_1{\mathbf{x}}_n\choose\pi_2\mathbf{x}_n})$ is the (unique) classical solution to $(\mathrm{ACP})_{\mathbf{x}_n}$ for each $n\in\mathbb{N}$, and we finally obtain $\limn {\cu}(t,{\bf x}_n)=\limn \binom{u}{Lu} (t,{\pi_1{\bf x}_n\choose\pi_2\mathbf{x}_n})=0$ uniformly in $[0,t_0]$.
\end{proof}

In the same way we obtain mild solutions to \eqref{eq:AIBVP}.

\begin{corollary}\label{cor:2.9.}
Let $({\tilde\ca},D({\tilde\ca}))$ generate a strongly continuous semigroup $(\tilde\ct(t))_{t\geq0}$ on $\cx$. Then $u(t):=\pi_1\tilde\ct(\cdot)\binom{f}{g}$ yields the unique mild solution to $\mathrm{(AIBVP)}_{f,g}$ for all initial data $(f,g)\in\mathcal X$.
\end{corollary}

Sufficient conditions on the operators $A$, $B$, and $L$ implying
$\ca$ and $\tilde\ca$ to be generators, hence implying the
well-posedness of \eqref{eq:AIBVP}, will be given in the next section.

% --- SECTION 3 ---
\section{WELL-POSED INITIAL--BOUNDARY VALUE PROBLEMS}
\label{sec:3}

We first study the situation from Example~1.2, i.e., we consider the
case where $B=0$.

\begin{assumption}\label{assum:3.1.}
In order to obtain well-posedness of \eqref{eq:AIBVP},
we now impose the following.
\begin{enumerate}[label=\rm{(G$_{\arabic*}$)}, itemsep=2pt, parsep=0pt]
\item The restriction $A_0:=A\big|_{\ker(L)}: D(A_0):=\ker(L)\to X$ is closed and densely defined, and it has non-empty resolvent set.
\item The boundary operator $L:D(A)\to \dx$ is surjective.
\item The operator $\binom{A}{L}:D(A)\subseteq X \to X\times\dx$ is closed.
\end{enumerate}
\end{assumption}

The following key lemma is essentially due to Greiner
(\cite[Lemma~1.2 and Lemma~1.3]{[Gr87]}, see also \cite[Lemma~2.3]{[CENN01]}).

\begin{lemma}\label{lem:3.2.}
Let $\lambda\in\rho(A_0)$. Then the restriction $L\big|_{\ker(\lambda-A)}$ is invertible and its inverse
$$\cd_\lambda:=\left(L\big|_{\ker(\lambda-A)}\right)^{-1}:\partial
X\to \ker(\lambda-A),$$ 
called \textbf{Dirichlet operator}, is bounded. Moreover, for all 
$\mu\in\rho(A_0)$, $\cd_\mu$ is related to $\cd_\lambda$ by
\begin{equation}\label{eq:3.1}
\cd_\lambda=(I+(\mu-\lambda) R(\lambda,A_0))\cd_\mu.
\end{equation}
\end{lemma}

These Dirichlet operators play an important role in the following.

\begin{lemma}\label{lem:3.3.}
Let $A_0$ generate a strongly continuous semigroup $(T(t))_{t\geq0}$. For each $\lambda\in\rho(A_0)$ define a family of operators $(Q_\lambda(t))_{t\geq0}$ from $\dx$ to $X$ by
$$Q_\lambda(t)v:=(\lambda-A_0) \int_0^t T(s)\cd_\lambda v\;ds,
\qquad\enspace t\geq0,\; v\in \dx.$$
Then, for a given $t\geq0$, all the operators $Q_\lambda(t)$, 
$\lambda\in\rho(A_0)$, coincide, and will be denoted by $Q(t)$.
In particular, if $A_0$ is invertible, then 
\begin{equation}\label{eq:3.2}
Q(t)=\left(I_X-T(t)\right)\cd_0
\end{equation}
for all $t\geq0$.
\end{lemma}

\begin{proof}
Let $\lambda,\mu\in\rho(A_0)$, $v\in\dx$. Then
$$\begin{aligned}
 Q_\lambda(t)v&=(\lambda-A_0) \int_0^t T(s)\cd_\lambda v\;ds\\
 &{\buildrel \eqref{eq:3.1}\over =}(\lambda-A_0) \int_0^t T(s)\cd_\mu v\;ds
 +(\mu-\lambda)(\lambda-A_0)\int_0^t T(s)R(\lambda,A_0)\cd_\mu v\;ds\\
 &=(\lambda-A_0) \int_0^t T(s)\cd_\mu v\;ds
 +(\mu-\lambda)\int_0^t T(s)\cd_\mu v\;ds
=Q_\mu(t)v
\end{aligned}$$
by standard properties of strongly continuous semigroups, cf. \cite[Lemma~II.1.13 and Thm.~II.1.14]{[EN00]}). Integrating by parts we obtain \eqref{eq:3.2}.
\end{proof}

We now return to the operator matrix $\ca$ defined in \eqref{eq:2.1}.

\begin{lemma}
\label{lem:domain-factorization}
Let $\lambda \in \rho(A_0)$. Then
\[
D({\ca}) = \left\{ \binom{u}{v} \in D(A) \times \dx : L u = v \right\}
\]
coincides with
\begin{equation}
\label{eq:ch-def}
\ch := \left\{ \binom{u}{v} \in X \times \dx : u - \cd_\lambda v \in D(A_0) \right\}.
\end{equation}
Moreover, the identity
\begin{equation}
\label{eq:factorization}
\ca - \lambda = \begin{pmatrix} A_0 - \lambda & 0 \\ 0 & -\lambda \end{pmatrix}
\begin{pmatrix} I_X & -\cd_\lambda \\ 0 & I_{\dx} \end{pmatrix} =: \ca_\lambda \crl
\end{equation}
holds.
\end{lemma}

\begin{proof}
To show that $D(\ca)$ equals $\ch$ take a vector $\binom{u}{v} \in \cx$. Then, $\binom{u}{v} \in \ch$ if and only if $u - \cd_\lambda v \in \ker(L)$. But $\cd_\lambda v \in \ker(\lambda - A) \subseteq D(A)$, and hence $u - \cd_\lambda v \in D(A_0) = \ker(L)$ if and only if $u \in D(A)$ and $L(u - \cd_\lambda v) = L u - v = 0$, i.e., $L u = v$. Thus $D(\ca) = \ch$.

It is clear that $\ch = D(\ca_\lambda \crl)$ since $\binom{u}{v} \in D(\ca_\lambda \crl)$ if and only if
\[
\crl \binom{u}{v} = \binom{u - \cd_\lambda v}{v} \in D(\ca_\lambda) = D(A_0) \times \dx.
\]
To show \eqref{eq:factorization}, we take $\binom{u}{v} \in D(\ca_\lambda \crl)$ and obtain
\[
\ca_\lambda \crl \binom{u}{v} = \begin{pmatrix} A_0 - \lambda & 0 \\ 0 & -\lambda \end{pmatrix}
\binom{u - \cd_\lambda v}{v} = \binom{(A - \lambda)u}{-\lambda v} = (\ca - \lambda) \binom{u}{v}. \qedhere
\]
\end{proof}

\begin{remark}
\label{rem:factorization-remarks}
\begin{enumerate}
\item The factorization \eqref{eq:factorization} shows that $\ca - \lambda$ is a so-called \emph{one-sided coupled operator matrix}. For more details about this general theory, we refer the reader to \cite{[En99]}, \cite{[En01]}, or \cite{[KMN03]}.
\item Observe that, under the Assumptions~\ref{assum:3.1.}, and by Lemma~\ref{lem:2.3.} and \ref{lem:2.4.}, $(\ca, D(\ca))$ is closed and densely defined. Therefore, Theorem~\ref{thm:2.8.} implies that \eqref{eq:AIBVP} is well-posed if and only if $\ca$ generates a strongly continuous semigroup.
\end{enumerate}
\end{remark}

\begin{theorem}
\label{thm:ca-generator}
Under the Assumptions~\ref{assum:3.1.} the following conditions are equivalent.
\begin{enumerate}[label=(\alph*)]
\item The operator $(\ca, D(\ca))$ generates a strongly continuous semigroup on $\cx$.
\item The operator $(A_0, D(A_0))$ generates a strongly continuous semigroup on $X$.
\end{enumerate}
If (a) holds, the semigroup $(\ct(t))_{t \ge 0}$ generated by $\ca$ is given by
\begin{equation}
\label{eq:semigroup-formula}
\ct(t) = \begin{pmatrix} T(t) & Q(t) \\ 0 & I_\dx \end{pmatrix},
\end{equation}
for all $t \geq 0$, where $(T(t))_{t \geq 0}$ is the semigroup generated by $A_0$ and $(Q(t))_{t \geq 0}$ is the family of operators introduced in Lemma~\ref{lem:3.3.}. In the particular case of invertible $A_0$, $\ttt$ is given by
\begin{equation}
\label{eq:semigroup-explicit}
\ct(t) = \begin{pmatrix} T(t) & (I_X - T(t))\cd_0 \\ 0 & I_\dx \end{pmatrix}.
\end{equation}
\end{theorem}

\begin{proof}
(a) $\Rightarrow$ (b). Assume that $\ca$ is a generator. Observe first that $\rho(\ca) = \rho(A_0) \setminus \{0\}$. Indeed, it follows from \eqref{eq:factorization} that for all $\lambda \in \rho(A_0) \setminus \{0\}$ the resolvent is given by
\begin{equation}
\label{eq:resolvent-formula}
R(\lambda, \ca) = \begin{pmatrix} R(\lambda, A_0) & \frac{1}{\lambda}\cd_\lambda \\ 0 & \frac{1}{\lambda} \end{pmatrix}.
\end{equation}
Hence, its powers are
\[
R(\lambda,\ca)^n = \begin{pmatrix} R(\lambda,A_0)^n & * \\ 0 & \frac{1}{\lambda^n} \end{pmatrix} \qquad \text{for all } n \in \mathbb{N}.
\]
Therefore, $\Vert{R(\lambda,A_0)^n}$ can be dominated by $\Vert{R(\lambda,\ca)^n}\Vert$ and satisfies a Hille--Yosida estimate. The closedness of $A_0$ follows by the closedness of $\ca$. Finally, $D(A_0)$ is dense in $X$ by Assumption (G$_1$), and hence $A_0$ is a generator.

(b) $\Rightarrow$ (a). Assume first $A_0$ to be invertible, and define $\ct(t)$ as in \eqref{eq:semigroup-explicit}. We show that the family $(\ct(t))_{t \geq 0}$ is a strongly continuous semigroup and then verify that its generator is $\ca$.

To prove the semigroup property, take $t,s \in \pos$ and $u \in X$, $v \in \dx$. We then have
\begin{align*}
\ct(t)\ct(s)\binom{u}{v} &= \begin{pmatrix} T(t) & Q(t) \\ 0 & I_\dx \end{pmatrix}
\begin{pmatrix} T(s) & Q(s) \\ 0 & I_\dx \end{pmatrix} \binom{u}{v} \\
&= \binom{T(t)T(s)u + T(t)\cd_0 v - T(t)T(s)\cd_0 v + \cd_0 v - T(t)\cd_0 v}{v} \\
&= \binom{T(t+s)u + (I_X - T(t+s))\cd_0 v}{v} \\
&= \ct(t+s)\binom{u}{v}.
\end{align*}
Moreover,
\[
\ct(0) = \begin{pmatrix} T(0) & Q(0) \\ 0 & I_\dx \end{pmatrix} = \begin{pmatrix} I_X & 0 \\ 0 & I_\dx \end{pmatrix} = I_\cx.
\]

To show strong continuity of $(\ct(t))_{t \ge 0}$, i.e.,
\[
\lim_{t \to 0^+} \ct(t)\binom{u}{v} = \binom{\lim\limits_{t \to 0^+} [T(t)u + Q(t)v]}{v} = \binom{u}{v}
\]
for all $\binom{u}{v} \in \cx$, it suffices to verify that the operators $Q(t)$ converge to $0$ as $t \to 0^+$. This follows indeed from \eqref{eq:3.2} and from the boundedness of $\cd_0$.

Finally, in order to show that $\ca$ is actually the generator of $(\ct(t))_{t \geq 0}$, we note that for $\lambda$ with large real part the Laplace transform of $\ct(\cdot)$ coincides with the resolvent of its generator. In particular, and taking into account the explicit formula \eqref{eq:resolvent-formula} for the resolvent of $\ca$, we obtain for the Laplace transform of the upper right entry ${\mathcal T}(\cdot)_{12}$ of $\ct(\cdot)$ the equality
\[
\mathcal{L}({\mathcal T}(\cdot)_{12})(\lambda) = R(\lambda,\ca)_{12} = \frac{1}{\lambda}\cd_\lambda.
\]
On the other hand, the convolution theorem for the Laplace transform implies that $\mathcal{L}(Q(\cdot)v)(\lambda) = \frac{1}{\lambda}\cd_\lambda v$ for all $v \in \dx$ and $\lambda$ sufficiently large. This finally proves that $\mathcal{L}(Q(\cdot)v)(\lambda) = \mathcal{L}(\ct(\cdot)_{12}v)(\lambda)$ for all $v \in \dx$ and $\lambda$ large. The injectivity of the Laplace transformation then implies that $\ct(\cdot)_{12} = Q(\cdot)$.

Finally, the case of $0 \notin \rho(A_0)$ can be discussed as above, by rescaling arguments.
\end{proof}

\begin{corollary}
\label{cor:bounded-semigroups}
Under the Assumptions~\ref{assum:3.1.}, $A_0$ generates a bounded strongly continuous semigroup if and only if $\ca$ generates a bounded strongly continuous semigroup.
\end{corollary}

\begin{proof}
Assume, without loss of generality, $A_0$ to be invertible. Hence the formula \eqref{eq:semigroup-explicit} yields the semigroup generated by $\ca$, and the claim follows.
\end{proof}

\begin{proposition}
\label{prop:analytic-semigroups}
Under the Assumptions~\ref{assum:3.1.}, $\ca$ generates a bounded analytic semigroup on $\cx$ if and only if $A_0$ generates a bounded analytic semigroup on $X$.
\end{proposition}

This result may be shown by the same similarity and perturbation techniques of Proposition~\ref{prop:feedback-generator} below. However, the following more direct proof can also be used as suggested to us by the referee whom we thank for this remark.

\begin{proof}
Assume $A_0$ to generate a bounded analytic semigroup. By \cite[Thm.~II.4.6]{[EN00]}, this is equivalent to assume that (i) $A_0$ generates a bounded strongly continuous semigroup, and (ii) there exists a constant $C > 0$ such that
\begin{equation}
\label{eq:analytic-condition}
\Vert{s R(r+is, A_0)} \leq C
\end{equation}
for all $r > 0$ and $0 \neq s \in \mathbb{R}$.

It hence follows by Theorem~\ref{thm:ca-generator} and Corollary~\ref{cor:bounded-semigroups} that also $\ca$ generates a bounded strongly continuous semigroup. Moreover, since $\rho(\ca) = \rho(A_0) \setminus \{0\}$ as remarked in the proof of Theorem~\ref{thm:ca-generator}, the resolvent operator $R(r+is, \ca)$ is defined for $r > 0$ and $0 \neq s \in \mathbb{R}$. Taking into account \eqref{eq:3.1}, \eqref{eq:resolvent-formula}, and \eqref{eq:analytic-condition}, the resolvent of $\ca$ satisfies an estimate of the form
\begin{align*}
\Vert{s R(r+is, \ca)} &\leq \Vert{s R(r+is, A_0)} + \Vert{\cd_\mu} + \Vert{s R(r+is, A_0) \cd_\mu} + 1 \\
&\leq \mathcal{C} := C + \Vert{\cd_\mu} + C\Vert{\cd_\mu} + 1
\end{align*}
for all $r > 0$ and $0 \neq s \in \mathbb{R}$, and arbitrary $\mu \in \rho(A_0)$. It follows that also $\ca$ generates a bounded analytic semigroup. The converse implication can be proven likewise.
\end{proof}

An answer to the problem posed in Example~\ref{exa:1.2.} can now be given by simply using the variation of constants formula for semigroups (see \cite[§VI.7]{[EN00]}).

\begin{proposition}
\label{prop:inhomogeneous-solution}
Assume that the Assumptions~\ref{assum:3.1.} hold and that $A_0$ generates a strongly continuous semigroup, and take $f \in X$ and $g \in \dx$. If $\psi \in L^1(\pos, \dx)$, then the unique (mild, in the sense of \cite[Def.~VI.7.2]{[EN00]}) solution to the inhomogeneous problem \eqref{eq:iACP} is given by
\[
\cu : \pos \rightarrow \cx, \qquad \cu : t \mapsto \cu(t) := \ct(t) \binom{f}{g} + \int_0^t \ct(t-s) \binom{0}{\psi(s)}\,ds.
\]
In the particular case where $A_0$ is invertible, $\cu$ is given by
\begin{equation}
\label{eq:explicit-inhomogeneous}
\cu(t) = \binom{T(t)f + \cd_0 g - T(t)\cd_0 g + \int_0^t (\cd_0 \psi(s) - T(t-s)\cd_0 \psi(s))\,ds}{g + \int_0^t \psi(s)\,ds}.
\end{equation}
Moreover, if $f \in D(A)$, $f\big|_{\partial\Omega} = g$, and $\psi \in W^{1,1}(\pos, \dx)$, then $\cu$ is a classical solution to \eqref{eq:iACP}.
\end{proposition}

This result can now be applied to the inhomogeneous problem stated in Example~\ref{exa:1.2.}.

\begin{example}
\label{ex:heat-application}
We have to show that the setting introduced in Example~\ref{exa:1.2.} fits the hypotheses yielding the results presented in §§2--3.

Take as state and boundary space $X := L^2(\Omega)$ and $\dx := L^2(\partial\Omega)$, respectively. As operator $A$ we take the Laplacian
\[
A := \Delta \quad \text{with domain} \quad D(A) := \left\{ u \in H^{1/2}(\Omega) \cap H^2_{\text{loc}}(\Omega) : \Delta u \in L^2(\Omega) \right\}
\]
and as boundary operator $L$ the trace operator
\[
L u := u\big|_{\partial\Omega} \quad \text{with domain} \quad D(L) := D(A)
\]
as introduced in \cite[§§2.5-8]{[LM72]}. In particular, $A$ is defined on the maximal domain such that the traces of its elements exist as $L^2(\partial\Omega)$ functions. This is why $(A, D(A))$ can be called the \emph{maximal operator}.

With these definitions, \eqref{eq:iHE} takes the form \eqref{eq:iACP}. In order to apply Proposition~\ref{prop:inhomogeneous-solution} we have to show the well-posedness of \eqref{eq:AIBVP}. For this, it suffices, by Theorem~\ref{thm:2.8.}, that $\ca$ generates a strongly continuous semigroup on the product space $\cx$.

We first note that the operator $A_0$ is in this case nothing but the Dirichlet Laplacian $\Delta^D$, which, by standard results (see, e.g., \cite[Thm.~7.2.7]{[Pa83]}) generates an analytic semigroup on $X$. Surjectivity of $L$ follows from \cite[Thm.~2.7.4]{[LM72]}, see also \cite[Lemma~3.1]{[CENN01]}. Finally, the proof of the closedness of $\binom{A}{L}$ can be found in \cite[Lemma~3.2]{[CENN01]}.

The Dirichlet Laplacian on $L^2(\Omega)$ is invertible, and hence an explicit solution to \eqref{eq:iHE} is yielded by \eqref{eq:explicit-inhomogeneous}, where $(T(t))_{t \geq 0}$ is the heat semigroup and $\cd_0$ the operator that maps each $g \in D(A)$ into the unique harmonic function attaining boundary values $g$.

Summing up, the Assumptions~\ref{assum:3.1.} are satisfied and by Proposition~\ref{prop:analytic-semigroups} well-posedness of the homogeneous \eqref{eq:HE} is ensured. Existence and uniqueness of the solution to problem \eqref{eq:iHE} then follows from Proposition~\ref{prop:inhomogeneous-solution} under suitable assumptions on the inhomogeneous term $\psi$ (see \cite[§VI.7]{[EN00]}).
\end{example}

\begin{remark}
\label{rem:non-dissipative}
Although the Dirichlet Laplacian is dissipative, the operator $\ca$ is \emph{not} dissipative and the semigroup $\ttt$ generated by $\ca$ is not contractive on $\cx$ (but bounded indeed, by Remark~\ref{rem:factorization-remarks}). In \cite{[FGGR02]} it is shown how weighted norms and suitable feedbacks can be used in this context to obtain contractivity.
\end{remark}

\section{ABSTRACT BOUNDARY FEEDBACK SYSTEMS}
\label{sec:boundary-feedback}

We now return to the operator $\tilde{\ca}$ defined in \eqref{eq:2.2}. This means that we replace the inhomogeneous term $\psi$ in \eqref{eq:iHE} or in \eqref{eq:iACP} by a boundary feedback, i.e., we take
\[
\psi(t) := B u(t), \qquad t \geq 0,
\]
for some operator $B$ from $X$ into $\dx$. With this interpretation, we call \eqref{eq:AIBVP} an \emph{abstract boundary feedback system}. We combine the results obtained in the previous section and perturbation theory for semigroups. The simplest case in which $B$ is bounded follows directly from \cite[Thm. III.1.3]{[EN00]}.

\begin{proposition}
\label{prop:bounded-feedback}
Assume that $B : X \rightarrow \dx$ is bounded. Under the Assumptions~\ref{assum:3.1.} the operator $\tilde{\ca}$ generates a strongly continuous (respectively, analytic) semigroup on $\cx$ if and only if $A_0$ generates a strongly continuous (respectively, analytic) semigroup on $X$.
\end{proposition}

The case of unbounded feedback operators defined on $D(A)$ can be treated by using the results of \cite{[En99]} on one-sided coupled operator matrices. We briefly sketch the main steps. First, define
\[
B_\lambda := B \cd_\lambda : \dx \rightarrow \dx
\]
for all $\lambda \in \rho(A_0)$ (observe that $B_\lambda$ is well defined since $R(\cd_\lambda) \subseteq D(A) \subseteq D(B)$). Then, the following lemma holds.

\begin{lemma}
\label{lem:feedback-factorization}
For $\lambda \in \rho(A_0)$ the operator $\tilde{\ca} - \lambda$ can be factorized as
\begin{equation}
\label{eq:feedback-factorization}
\tilde{\ca} - \lambda = \begin{pmatrix} A_0 - \lambda & 0 \\ B & B_\lambda - \lambda \end{pmatrix}
\begin{pmatrix} I_X & -\cd_\lambda \\ 0 & I_{\dx} \end{pmatrix}
\end{equation}
with $D(\tilde{\ca}) = \ch$ as defined in \eqref{eq:ch-def}.
\end{lemma}

\begin{proof}
Taking into account the identity \eqref{eq:factorization}, a matrix computation yields
\begin{align*}
\tilde{\ca} - \lambda &= \ca_\lambda \crl + \begin{pmatrix} 0 & 0 \\ B & 0 \end{pmatrix}
= \left[ \ca_\lambda + \begin{pmatrix} 0 & 0 \\ B & 0 \end{pmatrix} \crl^{-1} \right] \crl \\
&= \left[ \begin{pmatrix} A_0 - \lambda & 0 \\ 0 & -\lambda \end{pmatrix} + \begin{pmatrix} 0 & 0 \\ B & B_\lambda \end{pmatrix} \right] \crl
= \begin{pmatrix} A_0 - \lambda & 0 \\ B & B_\lambda - \lambda \end{pmatrix}
\begin{pmatrix} I_X & -\cd_\lambda \\ 0 & I_{\dx} \end{pmatrix}.
\end{align*}
Note that the domain of $\tilde{\ca} - \lambda$ coincides with $\ch$ from \eqref{eq:ch-def}.
\end{proof}

\begin{proposition}
\label{prop:feedback-generator}
Let the Assumptions~\ref{assum:3.1.} hold, and assume that $B$ is relatively $A_0$-bounded, and that $B_{\lambda_0}$ is bounded for some $\lambda_0 \in \rho(A_0)$. Then the operator $\tilde{\ca}$ generates a strongly continuous (resp., analytic) semigroup on $\cx$ if and only if the operator $A_0 - \cd_{\lambda_0} B$, defined on $D(A_0)$, generates a strongly continuous (resp., analytic) semigroup on $X$.
\end{proposition}

\begin{proof}
By \eqref{eq:feedback-factorization}, we have
\[
\tilde{\ca} = \begin{pmatrix} A_0 & 0 \\ B & B_{\lambda_0} \end{pmatrix} \mathcal R_{\lambda_0}
+ \lambda_0 \begin{pmatrix} 0 & \cd_{\lambda_0} \\ 0 & 0 \end{pmatrix} =: \tilde{\ca}_{\lambda_0} \mathcal R_{\lambda_0} + \lambda_0 \cp_{\lambda_0}.
\]
By Lemma~\ref{lem:3.2.}, $\cp_{\lambda_0}$ is a bounded perturbation, so it suffices to show that $\tilde{\ca}_{\lambda_0} \mathcal R_{\lambda_0}$ generates a strongly continuous (resp., analytic) semigroup. Due to the invertibility of $\mathcal R_{\lambda_0}$, $\tilde{\ca}_{\lambda_0} \mathcal R_{\lambda_0}$ is similar to
\[
\mathcal R_{\lambda_0} \tilde{\ca}_{\lambda_0} = \begin{pmatrix} A_0 - \cd_{\lambda_0} B & -\cd_{\lambda_0} B_{\lambda_0} \\ B & B_{\lambda_0} \end{pmatrix}
= \begin{pmatrix} A_0 - \cd_{\lambda_0} B & 0 \\ B & 0 \end{pmatrix} + \begin{pmatrix} 0 & -\cd_{\lambda_0} B_{\lambda_0} \\ 0 & B_{\lambda_0} \end{pmatrix} =: \mathcal{M} + \mathcal{N}.
\]
Since $B_{\lambda_0}$ is bounded, also $-\cd_{\lambda_0} B_{\lambda_0}$ and $\mathcal{N}$ are bounded. Hence, by bounded perturbation and similarity arguments, we can conclude that $\tilde{\ca}$ is a generator if and only if the operator matrix $\mathcal{M}$ is a generator.

Observe further that $B$ is also $(A_0 - \cd_{\lambda_0} B)$-bounded and while the operator matrix $\tilde{\ca}$ does not have diagonal domain, $D(\mathcal{M}) = D(\tilde{\ca}_{\lambda_0}) = D(A_0) \times \dx$: hence, by well-known results on matrices with diagonal domain (see, e.g., \cite[Cor.~3.2 and Cor.~3.3]{[Na89]}) one can see that $\mathcal{M}$ generates a strongly continuous (respectively, analytic) semigroup on $\cx$ if and only if $A_0 - \cd_{\lambda_0} B$ does the same on $X$.
\end{proof}

\begin{remark}
\label{rem:relative-boundedness}
Observe that, in particular, if $B$ is relatively $\binom{A}{L}$-bounded, then $B$ is relatively $A_0$-bounded and $B_{\lambda_0}$ is bounded for all $\lambda_0 \in \rho(A_0)$. Indeed, $\Vert{u}_{\binom{A}{L}} = \Vert{u}_A$ for all $u \in D(A_0)$, and hence $B$ is relatively $A_0$-bounded.

Further, the closedness of $\binom{A}{L} : D(A) \rightarrow \cx$ implies the closedness of the operator matrix
\[
\mathcal{L} := \begin{pmatrix} A & 0 \\ L & 0 \end{pmatrix}, \qquad D(\mathcal{L}) := D(A) \times \dx,
\]
on $\cx$, and hence we can define the Banach space $\mathcal{Y} := (D(\mathcal{L}), \Vert{\cdot}_{\mathcal{L}})$, with $\mathcal{Y} \hookrightarrow \cx$. Now observe that, by virtue of Lemma~\ref{lem:3.2.}, the operator
\[
\mathcal{H}_\lambda := \begin{pmatrix} 0 & \cd_\lambda \\ 0 & 0 \end{pmatrix} : \cx \rightarrow \cx
\]
is bounded from $\cx$ to $\cx$ and its range is contained in $\mathcal{Y}$ for all $\lambda \in \rho(A_0)$. It then follows by \cite[Cor.~B.7]{[EN00]} that $\mathcal{H}_\lambda$ is also bounded from $\cx$ to $\mathcal{Y}$, and since under the above assumptions
\[
\mathcal{B} := \begin{pmatrix} 0 & 0 \\ B & 0 \end{pmatrix}, \qquad D(\mathcal{B}) := D(A) \times \dx
\]
is bounded from $\mathcal{Y}$ to $\cx$,
\[
\mathcal{B} \mathcal{H}_\lambda = \begin{pmatrix} 0 & 0 \\ 0 & B_\lambda \end{pmatrix} : \cx \to \cx
\] 
is bounded, and the claim follows.
\end{remark}

\begin{example}[Diffusion-transport equation with dynamical boundary conditions]
\label{ex:diffusion-transport}
We consider the diffusion-transport equation with dynamical boundary conditions
\begin{equation}
\label{eq:DT}
\begin{cases}
\dot{u}(t,x) = u''(t,x) + k u'(t,x), & t \geq 0,\; x \in (0,1), \\
\dot{u}(t,0) = u'(t,0) + c u(t,0), & t \geq 0, \\
\dot{u}(t,1) = -u'(t,1) + d u(t,1), & t \geq 0, \\
u(0,x) = f(x), & x \in (0,1), \\
u(0,0) = a, \quad u(0,1) = b,
\end{cases}
\end{equation}
where $a,b,c,d \in \mathbb{C}$, $k \geq 0$. This system becomes an \eqref{eq:AIBVP} if we consider the state space $X := L^2(0,1)$, the boundary space $\dx := \mathbb{C}^2$, the operator
\[
A u := u'' + k u', \qquad D(A) := H^2(0,1),
\]
the boundary operator
\[
L u := \binom{u(0)}{u(1)}, \qquad D(L) := D(A),
\]
and the feedback operator
\[
B u := \binom{u'(0) + c u(0)}{-u'(1) + d u(1)}, \qquad D(B) := D(A).
\]

We first show that the Assumptions (G$_1$)--(G$_3$) hold. Observe that $L$ is surjective onto $\mathbb{C}^2$. The operator $A_0$, defined as the restriction of $A$ to $D(A_0) := H^2(0,1) \cap H^1_0(0,1)$, generates an analytic semigroup and, in particular, (G$_1$) is satisfied. Finally, one can show as in \cite[Lemma~3.3]{[CENN01]} that the operator $\mathcal{L} := \binom{\frac{d^2}{dx^2}}{L}$ is closed on $D(A)$. Since $\binom{k\frac{d}{dx}}{0}$ is relatively $\mathcal{L}$-bounded with $\mathcal{L}$-bound 0 for $k \geq 0$, their sum $\binom{A}{L}$ is also closed (see \cite[Lemma~III.2.4]{[EN00]}). Hence, the Assumptions~\ref{assum:3.1.} are satisfied and we are now in the position to apply Proposition~\ref{prop:feedback-generator}.

Observe that $0 \in \rho(A_0)$, and the Dirichlet operator $\cd_0$ maps, by definition, each pair $(\alpha,\beta) \in \mathbb{C}^2$ into the unique solution of the ordinary differential equation
\[
\begin{cases}
u''(x) + k u'(x) = 0, & x \in (0,1), \\
u(0) = \alpha, \quad u(1) = \beta.
\end{cases}
\]

In order to show that Proposition~\ref{prop:feedback-generator} applies, we now check that $B$ is relatively $\binom{A}{L}$-bounded. This will then yield, by virtue of Remark~\ref{rem:relative-boundedness}, that $B$ is relatively $A_0$-bounded, and that $B_\lambda$ is bounded for all $\lambda \in \rho(A_0)$. Recall that the first derivative on $L^2(0,1)$ is relatively bounded by the second derivative, with relative bound 0 (see, e.g., \cite[§III.2]{[EN00]}), and hence that the graph norm of the second derivative and the graph norm of $A$ are equivalent. It then follows from the embedding $H^1(0,1) \hookrightarrow C([0,1])$ that we can find suitable constants $\xi$, $\xi_1$, $\xi_2$ such that
\begin{align*}
\Vert{B u} &= \left( |u'(0) + c u(0)|^2 + |u'(1) - d u(1)|^2 \right)^{1/2} \\
&\leq \zeta \left( \Vert{u'}_{C([0,1])} + |L u|_{\mathbb{C}^2} \right) \\
&\leq \xi \left( \Vert{u''}_{L^2(0,1)} + \Vert{u'}_{L^2(0,1)} \right) + \zeta |L u|_{\mathbb{C}^2} \\
&\leq \xi_1 \Vert{u''}_{L^2(0,1)} + \xi_2 \Vert{u}_{L^2(0,1)} + \zeta |L u|_{\mathbb{C}^2} \leq \eta \Vert{u}_{\binom{A}{L}},
\end{align*}
for all $u \in D(A)$, where $\zeta = 2\max\{1, |c|, |d|\}$ and $\eta = \max\{\xi_1, \xi_2, \zeta\}$. Hence, $B$ is relatively $\binom{A}{L}$-bounded.

We finally show that $A_0 - \cd_0 B$ generates an analytic semigroup. The boundedness of $B$ from $(D(A_0), \Vert{\cdot}_{A_0})$ to $\mathbb{C}^2$ shows that $B$, and hence $\cd_0 B$ is relatively $A_0$-compact. It then follows from perturbation theory (see \cite[Cor.~III.2.17]{[EN00]}) that $A_0 - \cd_0 B$, defined on $D(A_0)$, is the generator of an analytic semigroup.

Therefore, the operator
\[
\ca := \begin{pmatrix} A & 0 \\ B & 0 \end{pmatrix}, \qquad D(\ca) := \left\{ \binom{u}{L u} : u \in D(A) \right\}
\] 
on $\cx := X \times \dx$ generates an analytic semigroup by Proposition~\ref{prop:feedback-generator}, and the diffusion-transport equation \eqref{eq:DT} is well-posed.
\end{example}

\begin{remark}
\label{rem:stability}
A question that naturally arises is how the stability of such a system is affected by the feedback operator. Under certain assumptions on the parameters $c$ and $d$ the semigroup generated by $\ca$ is positive, and it can be shown that \emph{stability of our boundary feedback system is independent of the internal diffusion}. This means that the spectral bound of $\ca$ is negative if and only if the spectral bound of $B_0$ is negative. This allows to apply Liapunov-type stability results (see \cite[Chap.~V]{[EN00]}), but we refer the reader to \cite[§§7--9]{[KMN03]} for details.
\end{remark}

\end{document}